\documentclass[12pt]{amsart}
\usepackage{amsbsy,amssymb,amscd,amsfonts,latexsym,amstext,delarray,amsmath,graphicx,color,caption,blkarray,mathtools}
\usepackage{tikz}
\usetikzlibrary{cd}
\usetikzlibrary{matrix,arrows}
\usepackage{graphicx}
\usepackage{hyperref}
\usepackage{float}
\input xy

\xyoption{all}
\usepackage{euscript}
\usepackage{multirow,charter}
\usepackage{etex, pictexwd,dcpic}
\usepackage[normalem]{ulem}

\newtheorem {theorem}    {Theorem}[section]

\newtheorem {lemma}      [theorem]    {Lemma}

\newtheorem {proposition}[theorem]    {Proposition}

\theoremstyle{definition}

\theoremstyle{remark}

\DeclareMathOperator{\SL}{SL}
\DeclareMathOperator{\GL}{GL}
\DeclareMathOperator{\PGL}{PGL}
\DeclareMathOperator{\gl}{\mathfrak{gl}}
\DeclareMathOperator{\Aut}{Aut}
\DeclareMathOperator{\Inn}{Inn}
\DeclareMathOperator{\End}{End}
\DeclareMathOperator{\Exp}{Exp}
\DeclareMathOperator{\Der}{Der}
\DeclareMathOperator{\Ad}{Ad}
\DeclareMathOperator{\ad}{ad}
\DeclareMathOperator{\Id}{Id}
\DeclareMathOperator{\htt}{ht}
\DeclareMathOperator{\re}{{re}}
\DeclareMathOperator{\im}{{im}}
\DeclareMathOperator{\rad}{{rad}}
\DeclareMathOperator{\Hom}{{Hom}}
\DeclareMathOperator{\grHom}{{grHom}}
\DeclareMathOperator{\grEnd}{{grEnd}}
\DeclareMathOperator{\grDer}{{grDer}}
\DeclareMathOperator*{\freeprod}{\raisebox{-0ex}{\scalebox{1.5}{$\Asterisk$}}}

\renewcommand{\a}{\alpha}
\renewcommand{\b}{\beta}
\newcommand{\g}{\gamma}

\newcommand{\F}{{\mathbb F}}
\newcommand{\Z}{{\mathbb Z}}
\newcommand{\R}{{\mathbb R}}
\newcommand{\C}{{\mathbb C}}
\newcommand{\N}{{\mathbb N}}
\newcommand{\Q}{{\mathbb Q}}
\newcommand{\M}{{\mathbb M}}

    \evensidemargin \oddsidemargin
\begin{document}

\title{Symmetries of Borcherds algebras}
\author{Lisa Carbone}

\begin{abstract} We give an overview of the construction of Borcherds algebras, particularly the Monstrous Lie algebras $\mathfrak m_g$ constructed by Carnahan, where $g$ is an element of the Monster finite simple group. When $g$ is the identity element, $\mathfrak m_g$ is the Monster Lie algebra of Borcherds. We discuss the appearance of the $\mathfrak m_g$ in compactified models of the Heterotic String. We also summarize recent work on associating Lie group analogs to the Lie algebras $\mathfrak m_g$. We include a discussion of some open problems.
\end{abstract}

\thanks{}

\maketitle


\section{Overview}
 Let $\mathbb{M}$ be the Fischer--Griess Monster finite simple group. Let $J(q)$ be the normalized elliptic modular invariant 
$$J(q ) = \sum_{n= -1}^\infty c(n)q^n=\dfrac{1}{q} + 196884q +21493760q^2+864299970q^3+\dots,
$$
where  $q=e^{2\pi{\bf{i}}\tau}$ for $\tau\in\C$ with $\text{Im}(\tau)>0.$
Building upon an earlier conjecture of McKay and Thompson,  Conway and Norton \cite{CN} formulated the Monstrous Moonshine Conjecture which states that  there exists a natural infinite-dimensional $\Z$-graded  $\M$-module  $W=\bigoplus_{n=-1}^\infty W_n$ such that  $\sum_{n=-1}^\infty\dim(W_n)q^n = J(q)$ and  for  every $g \in \M$, the generating function, or  McKay--Thompson series, $\sum_{n=-1}^\infty{\text{\rm tr}}\left(g|_ 
{W_n} \right)q^n, $
is  the Hauptmodul of a genus zero function field arising from a suitable discrete subgroup of $\SL(2,\R)$.

 Frenkel, Lepowsky and Meurman \cite{FLMPNAS} and \cite{FLM} constructed an infinite-dimensional $\Z$-graded $\M$-module, $V^\natural$, whose graded dimension is~$J(q)$.
They constructed  $V^{\natural}$ as a vertex operator algebra, called  the {\it Moonshine Module}  whose group of  automorphisms as a  vertex operator algebra is $\M$.

The Moonshine module $V^{\natural}$  was constructed as the sum of two subspaces $V_+$ and $V_-$,
which are the $+1$ and $-1$ eigenspaces of a particular involution in $\M$. If
an element $g\in\M$ commutes with this involution, then  \cite{FLM} showed that its  McKay--Thompson series $\sum_{n=-1}^\infty{\text{\rm tr}}\left(g|_{V_{n+1}^{\natural}}\right)q^n$ can be determined as the sum of two series given
by its traces on~$V_+$ and~$V_-$,  thus partially settling the Conway-Norton conjecture for $W=V^\natural$. For a detailed history we refer the reader to \cite{FLM}.

 Borcherds  \cite{BoInvent} used the construction of $V^\natural$ by \cite{FLM}  and the Monster Lie algebra $\mathfrak m$, which was discovered as part of a new class of Lie algebras, now known as  Borcherds algebras \cite{BoGKMA}, together with the No-ghost Theorem from String Theory, to prove the remaining cases of the Monstrous Moonshine conjecture.

Prior to the construction of $\mathfrak m$, Borcherds discovered the Fake Monster Lie algebra in the 1980s \cite{BorcherdsICM}, which served as a precursor for proving the Monstrous Moonshine conjecture. Borcherds constructed it using vertex operators associated with a Lorentzian lattice built from the Leech lattice, initially hoping this  might  yield the sought-after algebra connecting the Monster group to modular functions. However, instead of the Monster group, this algebra possessed symmetry related to the Conway group, and its denominator formula was an automorphic product which   specializes to the modular discriminant $\Delta$  rather than the $j$-function. 
The denominator
formula for the Fake Monster Lie algebra motivated Borcherds to construct the 
Monster Lie algebra, which he used to prove the Monstrous Moonshine conjecture \cite{BoInvent}.

Let $V_{{1,1}}$ denote the vertex  algebra  for the even unimodular  2-dimensional Lorentzian lattice  $\textrm{II}_{1,1}$ and let $V=V^\natural\otimes V_{{1,1}}$. Borcherds
constructed  the Monster Lie algebra $\mathfrak m$ in two different ways:

\begin{enumerate} 
\item   $\mathfrak m=P_1/R$  where $P_1$ is  the {\emph{physical space}} of $V$:
$$P_1=\{\psi\in V^\natural\otimes V_{1,1}\mid L(0)\psi=\psi,\ L(j)\psi=0, \ j> 0\}$$
which is the subspace of $V=V^\natural\otimes V_{1,1}$ consisting of primary vectors of weight 1, where the $L(j)$ are the Virasoro generators 
and $R$ denotes the radical of a natural bilinear form on $P_1$. 

\item $\mathfrak g(A)$ is the Lie algebra associated to  $A$,  a certain infinite Borcherds Cartan matrix. Then $\mathfrak m\cong\mathfrak{g}(A)/\mathfrak{z},$ where $\mathfrak{z}$ is the center of~$\mathfrak g(A)$. 
\end{enumerate}

Borcherds \cite{BoInvent} computed a  denominator formula for $ \mathfrak{m} $ and extended the homological methods of Garland and Lepowsky \cite{GarLep} to derive the replication formulas conjectured by Conway and Norton.

Conway and Norton suggested in \cite{CN} that the phenomena leading to Moonshine might be observed in groups other than $\M$. Further evidence for this claim was provided by Queen \cite{Queen}. In particular, Queen decomposed the coefficients of McKay--Thompson series into representations of quotients of subgroups of the Monster in various cases.

Norton then formulated the Generalized Moonshine conjecture  \cite{No,No2}.   To each commuting pair of elements $(g, h)$, the conjecture  proposes a holomorphic function $Z(g, h, \tau)$ (generalized McKay--Thompson series $T_{g,h}(\tau)$) on the upper half-plane which is either a constant or a Hauptmodul,  and which has a number of additional properties. The Moonshine Module $V^\natural$ is replaced   by a family of irreducible $g$-twisted modules $ V^\natural_g$, one  for each $g \in \mathbb{M}$, 
where each $V^\natural_g$ is  a  graded  representation of a central extension
 of the centralizer $C_{\mathbb{M}}(g)$.

The graded dimensions of
$V^\natural_g$
are given by functions $J_g(\tau)$ related to the generalized McKay--
Thompson series $T_g$ (taking $g=h$)
$$T_g(\tau)=\sum_{n=-1}^\infty{\text{\rm Tr}}\left(g\mid_ 
{V^\natural_{g, n+1}} \right)q^n $$ via the involution $\tau\to -1/\tau$ 
$$J_g(\tau)=T_g(-1/\tau).$$
Here $q=e^{2\pi{\bf{i}}\tau}$ and    $\tau\in\C$ with $\text{Im}(\tau)>0.$

Dong, Li and Mason \cite{DLiM}, H\"ohn \cite{Hohn} and Carnahan \cite{CarGMI}, \cite{CarDuke} proved many aspects of the Generalized Moonshine Conjecture.
For example, Dong, Li and Mason \cite{DLM} proved the existence of 
irreducible $g$-twisted modules
$V^\natural_g$, which are  unique up to isomorphism.  Such a module $V^\natural_g$ admits a
canonical projective action of $C_{\M}(g)$.
They  gave a construction of $V^\natural_{2A}$ (the baby Monster) for conjugacy class $2A$\footnote{in the notation of the Atlas of Finite Groups \url{https://brauer.maths.qmul.ac.uk/Atlas/v3/}} of $\M$ (also given by  H\"ohn \cite{Hohn}.

The construction of $g$-twisted $V^\natural$-modules  is a deep problem and involves the work of many people. In \cite{vEMS}, the authors developed aspects of the theory of  cyclic orbifolds, using an application of the Verlinde formula \cite{Hu2008}, the  tensor category theory of \cite{HL}
and the fixed-point regularity results of \cite{CM}.  The notion of twisted modules for lattice vertex operator algebras was first discovered by Frenkel, Lepowsky and Meurman \cite{FLM} and was the first example of what physicists call an orbifold conformal field theory.

\subsection{Monstrous Lie algebras}
 As part of his work on the Generalized Moonshine Conjecture, Carnahan \cite{CarGMI}, \cite{CarDuke}  constructed a family of Lie algebras $\mathfrak m_g$, one 
 for each conjugacy class $[g] \in \mathbb{M}$, associated to the irreducible $g$-twisted modules
$V^\natural_g$.

An element $g\in\mathbb{M}$ is called {\it Fricke} if  the McKay--Thompson series
$$T_{g}(\tau)=\sum_{n=-1}^\infty{\text{\rm Tr}}\left(g\mid_ 
{V^\natural_{n+1}} \right)q^n $$
is invariant
under the level $N$ Fricke involution $\tau \mapsto -1/N\tau$ for some $N\geq 1$. In this case, we say that $\mathfrak m_g$ has {\it Fricke type}, otherwise $\mathfrak m_g$ has {\it non-Fricke type}.

Carnahan showed that the Lie algebras $\mathfrak m_g$ are Borcherds algebras and that the $\mathfrak m_g$ admit a canonical  action of $C_{\M}(g)$ by homogeneous Lie algebra automorphisms. The  Lie algebras $\mathfrak m_g$ have the property that their denominator identities contain an infinite product expansion of the automorphic function arising from the McKay--Thompson series corresponding to the element $g\in\M$.

Carnahan's analysis of the denominator formula for the Lie algebras $\mathfrak m_g$ naturally breaks up into separate consideration of the cases where $\mathfrak m_g$ has Fricke type and  $\mathfrak m_g$ has non-Fricke type \cite{CarDuke}. These cases correspond to Fricke and non-Fricke elements $g\in\M$ \cite{CarFricke}.

The Lie algebras $\mathfrak m_g$ were discovered independently by I. Frenkel \cite{DF}. Duncan and Frenkel described the graded dimensions of Verma modules for the $\mathfrak m_g$ in terms of Rademacher sums.
They also conjectured a natural geometric interpretation of twisted
monstrous Lie algebras $\mathfrak m_g$ in the context of (conjectural) $g$-twisted chiral 3-dimensional quantum gravity theories.

 \subsection{Monstrous Lie algebras in string theory}

The appearance of Borcherds algebras  as symmetries of compactified models of the Heterotic String was first noted by Harvey and Moore (\cite{HM1}, \cite{HM2}).

More recently, Paquette, Persson, and Volpato (\cite{PPV1} and \cite{PPV2}) proposed a  physical interpretation of the Monstrous Moonshine
Conjecture in terms of a certain family of compactifications, labeled by conjugacy classes of $\mathbb{M}$, of the $E_8\times E_8$ Heterotic string to $1 + 1$ dimensions.

For each element $g\in \mathbb M$,   they consider a certain
Heterotic String compactification  to $1+1$ spacetime dimensions.  Each such model has  Carnahan's Monstrous Lie algebra $\frak m_g$ as an algebra
of spontaneously broken gauge symmetries.

In any supersymmetric theory, a BPS-state is a physical state which preserves some of the supersymmetry. In \cite{PPV1}, the authors  introduced a supersymmetric index $Z$ which counts (with signs) the number
of BPS-states. The space of BPS-states contributing to the supersymmetric index is naturally
isomorphic to
 $\bigwedge^\infty\frak m^-$.

It was shown in \cite{HM1}  that BPS states in string theory form an algebra, 
though a complete understanding of the algebra of BPS states of the Heterotic String is still lacking. In the Heterotic String constructions discussed presently, however, \cite{PPV1} and \cite{PPV2} realized the Hilbert space of BPS states $\mathcal H$ as a module for
the Monster Lie algebra $\frak m$ when $g=1$ (the analogous statement for $\frak m_g$ holds when $g \neq 1$), while Harvey and Moore showed that there is an algebraic
structure on $\mathcal H$ itself.

Generalized Moonshine also has an interpretation in physics, proposed by Dixon, Ginsparg and Harvey \cite{DGH}. In  particular, they interpreted the $V^\natural_g$ as twisted sectors of a conformal field theory with $\M$-symmetry and interpreted the functions $Z(g, h, \tau)$ as genus one partition functions.




\section{Borcherds algebras}
Our exposition closely follows \cite{BoInvent}, \cite{JurJPAA}, \cite{JurContempMath} and \cite{JLW}. Let $I$ be a countable index set.  Let $A=(a_{ij})$ with $a_{ij}\in\R$ and $i,j\in I$.
Let
$$X=\{ e_i,\ f_i, \ h_i\mid  i\in I\}$$
where the $e_i$, $f_i$ and $h_i$ are formal symbols. Let $L(X)$ denote the free Lie algebra on $X$. We form the Lie algebra $\mathfrak{g}_0=\mathfrak{g}_0(A)$ over $\C$ by setting
$$\mathfrak{g}_0=L(X)/\mathcal{I}$$
where $\mathcal{I}$ is the ideal generated by the following Lie words

\begin{align*}
 [h_i, h_j]  &\quad\textrm{ for all $i,j\in I$},\\
[h_i,e_j]-a_{ij} e_j &\quad\textrm{ for all $i,j\in I$},\\
[h_i,f_j]+a_{ij} f_j &\quad\textrm{ for all $i,j\in I$},\\
[e_i,f_j]-\delta_{ij}h_i &\quad\textrm{ for all $i,j\in I$}.
\end{align*}
\bigskip

Consider the abelian group $\Z^I$ with pointwise addition of coordinates.
The relations generating $\mathcal{I}$ consist of
homogeneous elements, so $\mathfrak{g}_0=L(X)/\mathcal{I}$ is also a graded $\Z^I$ Lie algebra.

Let $A(X)$ be the free associative algebra on $X$. Then $A(X)$ has a $\Z^I$-grading determined by assigning degrees:
\begin{align*}
&\textrm{$h_i$ has degree $(0,\dots,0,\dots)$,}\\
&\textrm{$e_i$ has degree $(0,\dots,0,1,0,\dots,0,\dots)$,}\\
&\textrm{$f_i$ has degree $(0,\dots,0,-1,0,\dots,0,\dots)$, 
with $\pm 1$ in the $i$-th position for each $i$,}
\end{align*}

 When $A(X)$ is viewed as a Lie algebra under the Lie bracket $[u, v] = uv - vu$, $L(X)$ is the Lie subalgebra of $A(X)$ generated by $X$.
Restricting this grading to $L(X)$, we get a $\Z^I$-grading on $L(X)$:

\begin{align*}
    &\textrm{$\text{deg}[h_i,e_j]=\text{deg}(e_j)$,}\\
&\textrm{$\text{deg}[h_i,f_j]=\text{deg}(f_j)$,}\\
&\textrm{$\text{deg}[e_i,f_i]=\text{deg}(h_i)$,}\\
&\textrm{$\text{deg}[e_i,f_j]_{i\neq j}=(0,\dots,0,1,0,\dots,0,-1,0,\dots)$, with $1$ in the $i$-th position}\\
&\qquad\qquad\qquad\textrm{and $-1$ in the $j$-th position.}
\end{align*}

For $z\in \mathfrak{g}_0(A)$
$$\text{deg}(z+\mathcal{I})=\text{deg}(z)$$
unless $z\in\mathcal{I}$. 
 Let $(n_1,n_2,\dots)\in \Z^I$. Let $L(X)(n_1,n_2,\dots)$ denote the subspace of $L(X)$ with grading~$(n_1,n_2,\dots)$ and let $$\mathfrak{g}_0(n_1,n_2,\dots )=L(X)(n_1,n_2,\dots)/\mathcal{I}\cap L(X)(n_1,n_2,\dots)$$ denote the subspace of $\mathfrak{g}_0$ with grading~$(n_1,n_2,\dots)$. 

For each $i\in I$, we define a linear endomorphism $D_i$ of $L(X)$ where 

\smallskip
{\centerline{$D_i$ acts as the scalar $n_i$ on $L(X)(n_1,n_2\dots)$.}}

Then the $D_i$ are derivations called the {\it degree derivations} of $L(X)$. The $D_i$ are also naturally defined on subspaces $\mathfrak{g}_0(n_1,n_2,\dots)$  of $\mathfrak{g}_0$.

Let $\eta$ be the involution, or automorphism  of order two, on  $L(X)$ which takes 
\begin{align*} 
h_i&\mapsto -h_i,\\
e_i&\mapsto f_i,\\
f_i&\mapsto e_i.
\end{align*}
Then $\eta$ is also well-defined on  $\mathfrak{g}_0$.
We use the same symbols $e_i,\ f_i,\ h_i$ for the images of these elements in $\mathfrak{g}_0$.

Let $\mathfrak h$ be the abelian subalgebra generated by the $h_i$ for $i\in I$. Let $\mathfrak{g}_0^+$ be the Lie subalgebra generated by the $e_i$ for $i\in I$ and let $\mathfrak{g}_0^-$ be the Lie subalgebra generated by the $f_i$ for $i\in I$. 

\begin{theorem}\cite{JurContempMath} The Lie algebra $\mathfrak{g}_0$ has triangular decomposition
$$\mathfrak{g}_0=\mathfrak{g}_0^-\oplus\mathfrak h\oplus \mathfrak{g}_0^+.$$
The abelian subalgebra $\mathfrak h$ has basis consisting of the $h_i$ for $i\in I$. The subalgebras $\mathfrak{g}_0^\pm$ are free Lie algebras generated by the $e_i$, for $i\in I$, respectively $f_i$, for $i\in I$. 
\end{theorem}

As in \cite{JLW} and \cite{JurContempMath} we will   work with the extended Borcherds 
algebra $\mathfrak g^e$, that is, the Borcherds algebra with suitable degree derivations
adjoined.  In the extended Lie algebra, the Cartan subalgebra $\mathfrak h$ is  enlarged to make the simple roots
linearly independent and of multiplicity one. The extended Cartan subalgebra is denoted $\mathfrak h^e$.

Since $\mathfrak{g}_0=\mathfrak{g}_0^-\oplus\mathfrak h\oplus \mathfrak{g}_0^+$, we have (\cite{MP}, Section 1.8, Corollary 3)
$$\mathcal{U}(\mathfrak{g}_0)\cong\mathcal{U}(\mathfrak{g}_0^-)\otimes\mathcal{U}(\mathfrak{h})\otimes\mathcal{U}(\mathfrak{g}_0^+)$$
as $\C$-vector spaces.

Recall that the 3-dimensional Heisenberg Lie algebra $\mathcal{H}$
$$\left\{\begin{pmatrix}
 0 & a & c\\
 0 & 0 & b\\
 0 & 0 & 0\\
\end{pmatrix}\mid a,b,c\in\R\right\}$$
is the Lie algebra with basis
$$
X = \begin{pmatrix}
 0 & 1 & 0\\
 0 & 0 & 0\\
 0 & 0 & 0\\
\end{pmatrix},\quad
Y = \begin{pmatrix}
 0 & 0 & 0\\
 0 & 0 & 1\\
 0 & 0 & 0\\
\end{pmatrix},\quad
Z = \begin{pmatrix}
 0 & 0 & 1\\
 0 & 0 & 0\\
 0 & 0 & 0\\
\end{pmatrix}
$$
satisfying relations
$$[X, Y] = Z,\quad [X, Z] = 0,\quad [Y, Z] = 0.$$

\begin{proposition}\label{triple} Let $i\in I$ be fixed and let ${\mathfrak{u}}_{i}$ be the subalgebra of $\mathfrak{g}_0$ spanned by the elements $e_i$, $f_i$, $h_i$. 
\begin{enumerate}
\item Suppose that $a_{ii}\neq 0$.  Let 
$f=\left( \begin{smallmatrix}
 0& 0\\
1 & 0
\end{smallmatrix}
\right), \ h=\left( \begin{smallmatrix}
1 & ~0\\
0 & -1
\end{smallmatrix}\right),\ 
e=\left( \begin{smallmatrix}
 0& 1\\
0 & 0
\end{smallmatrix}
\right)
$
be the usual basis for $\mathfrak{sl}_2(\C)$.  Let $\widehat{e}_i=\dfrac{2e_i}{a_{ii}}$ and $\widehat{h}_i=\dfrac{2h_i}{a_{ii}}$. Then  the map

\begin{align*} 
\varphi:{\mathfrak{u}}_{i}&\to \mathfrak{sl}_2(\C)\\
f_i&\to f,\\
\widehat{h}_i&\to h,\\
\widehat{e}_i&\to e
\end{align*}
is an isomorphism of Lie algebras.

\item Suppose that $a_{ii}=0$. Then  the map

\begin{align*} 
\varphi:{\mathfrak{u}}_{i}&\to \mathcal{H}\\
f_i&\mapsto Y,\\
{h}_i&\mapsto Z,\\
{e}_i&\mapsto X
\end{align*}
is an isomorphism of Lie algebras.
\end{enumerate}
\end{proposition}

\begin{proof} For (1), from the defining relations for ${\mathfrak{g}_0}(A)$ we have
\begin{align*}
[\widehat{h}_i,\widehat{e}_i]&=[\dfrac{2h_i}{a_{ii}},\dfrac{2e_i}{a_{ii}}]=\left(\dfrac{2}{a_{ii}} \right)^2  [h_i,e_i]= \left(\dfrac{2}{a_{ii}} \right)^2a_{ii}e_i=\dfrac{4e_i}{a_{ii}} =2 \widehat{e}_i\\
&\text{so $\varphi([\widehat{h}_i,\widehat{e}_i])=[h,e]=\varphi(2\widehat{e}_i)=2e$;}\\
[\widehat{h}_i,f_i]&=\dfrac{2}{a_{ii}}[h_i,f_i]=-\dfrac{2}{a_{ii}} a_{ii}f_i=-2f_i\\
&\text{so $\varphi([\widehat{h}_i,{f}_i])=[h,f]=\varphi(-2{f}_i)=-2f$;}\\
[\widehat{e}_i,f_i]&=[\dfrac{2e_i}{a_{ii}},f_i]=\dfrac{2}{a_{ii}}[e_i,f_i]=\dfrac{2}{a_{ii}}h_i=\widehat{h}_i\\
&\text{so $\varphi([\widehat{e}_i,{f}_i])=[e,f]=\varphi(\widehat{h}_i)=h$}
\end{align*}
Thus ${\mathfrak{u}}_{i}\cong  \mathfrak{sl}_2(\C)$. 

For (2), we have
\begin{align*}
\varphi([{e}_i,f_i])&=[X,Y]=\varphi(h_i)=Z,\\
\varphi([e_i,h_i])&=[X, Z] =\varphi(a_{ii}e_i)=\varphi(0)=0,\\
\varphi([f_i,h_i])&=[Y, Z] =\varphi(-a_{ii}f_i)=\varphi(0)=0.
\end{align*}
So ${\mathfrak{u}}_{i}\cong  \mathcal{H}$.

\end{proof}

If $a_{ii}\neq 0$, we call  $\{\widehat{e}_i, {f}_i, \widehat{h}_i\}$ an {\it $\frak{sl}_2$-triple}.

If $a_{ii}= 0$, we call ${\mathfrak{u}}_{i}$ is generated by $\{{e}_i, {f}_i, {h}_i\}$ a {\it Heisenberg triple}.


\subsection{Borcherds Cartan matrices}
Let $I=\{1,2,3,\dots\}$. Let $A = (a_{ij})_{i,j \in I}$ be a
matrix with entries in ${\mathbb R}$, satisfying the following conditions:
\begin{enumerate}
\item[(B1)]  $A$ is symmetric.
\item[(B2)]  If $ i\neq j$ then $a_{ij}~\leq~0 $.
\item[(B3)]  If $a_{ii} > 0$ then $\frac{2a_{ij}}{  a_{ii}} \in \mathbb Z $
for all $j  \in I$. 
\end{enumerate}
  Let
$$X=\{ e_i,\ f_i, \ h_i\mid  i\in I\}$$
Let $L(X)$ denote the free Lie algebra on $X$. We form the Lie algebra $\mathfrak{g}_0=\mathfrak{g}_0(A)$ over $\C$ by setting
$$\mathfrak{g}_0=L(X)/\mathcal{I}$$
where $\mathcal{I}$ is the ideal generated by the following Lie words

\begin{align*}
&\textrm{$[h_i, h_j]$,  \quad for all $i,j\in I$,}\\
&\textrm{$[h_i,e_j]-a_{ij} e_j$, \quad for all $i,j\in I$,}\\
&\textrm{$[h_i,f_j]+a_{ij} f_j$, \quad for all $i,j\in I$,}\\
&\textrm{$[e_i,f_j]-\delta_{ij}h_i$, \quad for all $i,j\in I$.}
\end{align*}

The following theorem gives criteria for a Lie algebra to be a Borcherds algebra.

\begin{theorem} \cite{BoInvent, JurJPAA} Let $\mathfrak g$ be a Lie algebra satisfying the
following conditions:
\begin{enumerate}
 \item   $\mathfrak g$ can be $\mathbb Z$-graded as $\bigoplus_{i\in {\mathbb
Z}} \mathfrak g_i$, where $\mathfrak g_i$ is finite dimensional if $i \neq 0$, and
$\mathfrak g$ is diagonalizable with respect to $\mathfrak g_0$. 
\item$\mathfrak g$ has an involution $\omega$ which maps $\mathfrak g_i$
onto $\mathfrak g_{-i}$ and {acts as $-1$} on non-central elements of
$\mathfrak g_0$. In particular, $\mathfrak g_0$ is abelian.
\item  $\mathfrak g$ has a Lie algebra-invariant bilinear form $(\cdot,\cdot)$,
invariant under $\omega$, such that $\mathfrak g_i$ and $\mathfrak g_j$ are
orthogonal if $i \neq -j$, and such that the form $(\cdot,\cdot)_0$,
defined by $(x,y)_0 = -(x, \omega (y))$ for $x,y \in \mathfrak g$, is positive
definite on $\mathfrak g_m $ if $m \neq 0$.
\item$\mathfrak g_0 \subset [\mathfrak g,\mathfrak g]$. 
\end{enumerate}

Then $\mathfrak{g}$ is a Borcherds algebra, that is, there is a Borcherds Cartan matrix $A$ such that $\mathfrak{g}$ is isomorphic to $\mathfrak{g}(A)/\mathfrak{z}$ for some central subalgebra $\mathfrak{z}$. 
\end{theorem}

   We fix a distinguished subset $I_0 \subset I$ with $i \in I_0$ if $a_{ii} > 0$.
For all $i\neq j$ and $i\in I_0$ we define:
\begin{align*}
d_{ij}^+&=(\ad e_i)^{\frac{-2a_{ij}}{a_{ii}}+1}e_j\in\mathfrak{g}_0^+=\text{Lie}\{e_i\mid i\in I\}\\
d_{ij}^-&=(\ad f_i)^{\frac{-2a_{ij}}{a_{ii}}+1}f_j\in\mathfrak{g}_0^-=\text{Lie}\{f_i\mid i\in I\}.
\end{align*}

Let $\mathfrak{k}_0^+$ be the ideal of $\mathfrak{g}_0^+$ generated by

$$d_{ij}^+ {\text{ for $i\in I_0$, $j\in I$ and }}[e_i,e_j]{\text{ for $i,j\in I$ if }}a_{ij}=a_{ji}=0.$$

Let $\mathfrak{k}_0^-$ be the ideal of $\mathfrak{g}_0^-$ generated by

$$d_{ij}^- {\text{ for $i\in I_0$, $j\in I$ and }}[f_i,f_j]{\text{ for $i,j\in I$ if }}a_{ij}=a_{ji}=0.$$

Then $\mathfrak{k}_0^\pm$ are subalgebras of $\mathfrak{g}_0$. Let $\mathfrak{k}_0=\mathfrak{k}_0^+\oplus \mathfrak{k}_0^-$. The subalgebras $\mathfrak{k}_0^\pm$ and $\mathfrak{k}_0$ are ideals of $\mathfrak{g}_0$ \cite{JurContempMath}.
The Borcherds algebra $\mathfrak{g}=\mathfrak{g}(A)$ is defined to be 
$\mathfrak{g}=\mathfrak{g}_0/\mathfrak{k}_0$ \cite{JurContempMath}. Thus $\mathfrak{g}=\mathfrak{g}(A)$ is the Lie algebra with generators 
$$\{e_i,f_i,h_i\mid i\in I\}$$ and relations

\begin{align*}
&\textrm{$[h_i, h_j]=0$,   for all $i,j\in I$,}\\
&\textrm{$[h_i,e_j]=a_{ij} e_j$,  for all $i,j\in I$,}\\
&\textrm{$[h_i,f_j]=-a_{ij} f_j$,  for all $i,j\in I$,}\\
&\textrm{$[e_i,f_j]=\delta_{ij}h_i$,  for all $i,j\in I$,}\\
&\textrm{$(\ad e_i)^{\frac{-2a_{ij}}{a_{ii}}+1}e_j=0$, for all $i,j\in I$ with $i\neq j$ and $i\in I_0\subset I$,}\\
&\textrm{$(\ad f_i)^{\frac{-2a_{ij}}{a_{ii}}+1}f_j=0$, for all $i,j\in I$ with $i\neq j$ and $i\in I_0\subset I$,}\\
&\textrm{$[e_i,e_j]=0$  if $a_{ij}=a_{ji}=0$, for all $i,j\in I$, }\\
&\textrm{$[f_i,f_j]=0$ if $a_{ij}=a_{ji}=0$, for all $i,j\in I$.}
\end{align*}

 Since the ideal $\mathfrak{k}_0$ is homogeneous, $\mathfrak{g}$ has a
$\Z^I$-gradation induced by the given grading of $\mathfrak{g}_0$. We use the letters $\mathfrak h$ and  $e_i$, $f_i$, $h_i$ to denote their images in $\mathfrak{g}$. The involution $\eta$ sending
 \begin{align*} 
h_i&\mapsto -h_i,\\
e_i&\mapsto f_i,\\
f_i&\mapsto e_i.
\end{align*}
 is well-defined on $\mathfrak{g}$.

Let 
$$\mathfrak{n}^+=\mathfrak{g}_0^+/\mathfrak{k}_0^+,$$
$$\mathfrak{n}^-=\mathfrak{g}_0^-/\mathfrak{k}_0^-.$$

Recall the definition of degree derivations:

\smallskip
{\centerline{$D_i$ acts as the scalar $n_i$ on $\mathfrak{g}_0(n_1,n_2,\dots)$.}}

The $D_i$ are also defined  on subspaces $\mathfrak{g}(n_1,n_2,\dots)$ of $\mathfrak{g}$. Recall that these are the subspaces of $\mathfrak{g}$ with grading~$(n_1,n_2,\dots)\in \Z^I$. These subspaces also have the property that $\mathfrak{g}(n_1,n_2,\dots)=0$  unless only finitely many of the $n_i$ are non-zero. Given any $x\in\mathfrak{g}$ with grading~$(n_1,n_2,\dots)$, only finitely many of the $n_i$ are non-zero. In this case, $\mathfrak{n}^+$ is spanned by elements
$$[e_{i_1},[e_{i_2},[\dots,[e_{i_{r-1}},e_{i_r}]\dots]$$
where $e_j$ occurs $|n_j|$ times and $\mathfrak{n}^-$ is spanned by elements
$$[f_{i_1},[f_{i_2},[\dots,[f_{i_{r-1}},f_{i_r}]\dots]$$
where $f_j$ occurs $|n_j|$ times. The space $\mathfrak{g}(0,\dots,0,\pm 1,0,\dots)$ where $\pm 1$ occurs in the $i$-the position is spanned by $e_i$, respectively $f_i$.

\begin{lemma} (\cite{JurContempMath}) For each $i\in I_0$, $\mathfrak{g}$ is a direct sum of irreducible $\mathfrak{u}_i$-submodules, where $\mathfrak{u}_i=(\mathfrak{sl}_2)_i$.
\end{lemma}

\subsection{A symmetric bilinear form}
We define a symmetric bilinear form $(\cdot,\cdot)_Q$ on $Q$
taking values in $\R$  by  
$(\alpha_i, \alpha_j)_Q = a_{ij}$ and extending linearly to any element of  $Q$.  As usual, this form extends to a form on $\mathfrak{h}$ and its dual 
$\mathfrak{h}^{\ast}$ which satisfies the defining condition
 $$ ( \lambda, \alpha_i) = \lambda (h_i) \quad \text{for all $\alpha_i$
and $\lambda \in \mathfrak {h}^{\ast}$}. $$ We fix such a bilinear form. 
This may be extended to a symmetric bilinear form on $(\mathfrak{h}^e)^*$, as in \cite{JurContempMath}. This extension is not unique as it depends on a choice of symmetric bilinear form on $\mathfrak{d}$.

$$([a,b],c)_{\mathfrak g}=-(b,[a,c])_{\mathfrak g}$$

There is also an extension of the symmetric bilinear form on $(\mathfrak{h}^e)^*$ to a unique $\mathfrak g^e$-invariant symmetric bilinear form on $(\mathfrak{g}^e)^*$ (\cite{JurContempMath}, Theorem 1.6).



\subsection{Root system}

Let $\mathfrak{d}_0$ denote the (possibly
infinite-dimensional) space of commuting derivations of $\mathfrak{g}$ spanned by the $D_i$ for $i\in I$. Let $\mathfrak{d}$ be a subspace of $\mathfrak{d}_0$. Form the semi-direct product Lie algebra $\mathfrak{g}^e=\mathfrak{d}\rtimes \mathfrak{g}$. 
Then $\mathfrak{h}^e=\mathfrak{d}\rtimes \mathfrak{h}$ an abelian subalgebra of  $\mathfrak{g}^e$ which acts as scalar multiplication
on each space $\mathfrak{g}(n_1,n_2,\dots)$.

Let $\a_i\in (\mathfrak{h}^e)^*$, $i\in I$, be defined by the conditions:
$$[h,e_i]=\a_i(h)e_i\text{ for all }h\in \mathfrak{h}^e.$$
Then for all $i,j\in I$, we have 
\begin{align*}
\alpha_i(h_j) &= a_{ij}\\
D_i(e_j) &= \delta_{i,j}\\
D_i (f_j) &= -\delta_{i,j}\\
D_i (h) &= 0,\text{ for any $h\in \mathfrak{h}^e$}.
\end{align*}

\begin{lemma}\cite{JurContempMath} For $\mathfrak{d}=\mathfrak{d}_0$, the simple roots $\a_i$ are linearly independent in $(\mathfrak{h}^e)^*$.
\end{lemma}

For all $\alpha\in (\mathfrak{h}^e)^*$, we define
$$\mathfrak{g}_{\a}=\{x\in\mathfrak{g}\mid [h,x]=\a(x)\text{ for all }h\in \mathfrak{h}^e\}.$$
The root space $\mathfrak{g}_{\a}$ is contained in $\mathfrak{g}(A)$, but can be considered a subset of $\mathfrak{g}^e$
since $\mathfrak{g}_\a \cap \frak h ^e=0$. 

For all $\alpha,\beta\in (\mathfrak{h}^e)^*$
$$[\mathfrak{g}_\a,\mathfrak{g}_\beta]\subset \mathfrak{g}_{\a+\b}.$$

We have $e_i\in \mathfrak{g}_{\a_i}$, $f_i\in \mathfrak{g}_{-\a_i}$ for all $i\in I$. If all $n_i\geq 0$ or if all $n_i\leq 0$ and only finitely many of the $n_i$ are non-zero, then
$$\mathfrak{g}(n_1,n_2,\dots)\subset \mathfrak{g}_{n_1\a_1+n_2\a_2+\dots}.$$
$$\mathfrak{g}(n_1,n_2,\dots)=\mathfrak{g}_{n_1\a_1+n_2\a_2+\dots}$$
and $\mathfrak{g}{(0,0,\dots)}=\mathfrak h$.

The {\it roots} of $\mathfrak{g}$ are the non-zero $\alpha\in (\mathfrak{h}^e)^*$ such that $\mathfrak{g}_\a\neq \{0\}$. Let $\Delta$ be the set of roots. Let $\Delta^+$ be the set of positive roots, that is the roots $n_1\a_1+n_2\a_2+\dots$ for which $n_i\geq 0$ and only finitely many of the $n_i$ are non-zero. Let $\Delta^+$ be the set of negative roots, that is the roots $n_1\a_1+n_2\a_2+\dots$ for which $n_i\leq 0$ and only finitely many of the $n_i$ are non-zero. All roots lie on the root lattice $Q=\bigoplus_{i\in I}\Z\a_i$.

We call $\a\in\Delta$ a \emph{real root} if $(\alpha, \alpha)_Q>0$ and an \emph{imaginary root} if $(\alpha, \alpha)_Q\le 0$.  The set $I_0=\{i \in I\mid a_{ii}=(\alpha_i, \alpha_i )_Q >0\}$
indexes the real simple roots and $I\backslash I_0=\{i\in I\mid a_{ii}=(\alpha_i, \alpha_i )_Q\leq 0\}$ indexes the imaginary simple roots.

The set of real (respectively, imaginary) roots is denoted $\Delta^{\re}$ (respectively, $\Delta^{\im}$); 
we also denote the set of  positive real roots (respectively,  negative real roots) by $\Delta^{\re}_{\pm}$ and the set of  positive imaginary roots (respectively,   negative imaginary roots) by $\Delta^{\im}_{\pm}$.

Elementary properties of Borcherds algebras are summarized 
in the following, which follows from Proposition~1.5 of~\cite{JurContempMath}. 
\begin{proposition} \label{P-all}
The Borcherds algebra $\mathfrak{g}(A)$ has the following properties:
\begin{enumerate}
\item  $\mathfrak{g}_{n\a_i}=0$ unless $n=\pm 1$.
\item $\mathfrak{g}= \frak n^{-} \oplus \frak h \oplus \frak n^{+}$ where $\frak n^+ = \bigoplus_{\a \in \Delta_+} \frak g_\a$
   and $\frak n^-  = \bigoplus_{\a \in \Delta_-} \frak g_\a$.
\item $\left\{h_{i}\mid i\in I \right\}$ is a basis of ${\frak h}$.
\item $\Delta = \Delta_+ \cup \Delta_-$ (disjoint union).
\item $\left[\frak g_\a, \frak g_\b \right] \subset \frak g_{\a+\b}$ for all $\a,\b \in \left(\mathfrak{h}^e \right)^*$.
\item $\eta \left(\frak g_\a \right) = \frak g_{- \a}$ and $\dim \frak g_\a = \dim \frak g_{-\a} < \infty$ for all $\a \in (\frak h^e)^*$.
\item If the (possibly infinite) matrix
$A'$ 
is obtained from the matrix  $A$ by a
permutation of the rows and a corresponding permutation of columns, then
there is a natural isomorphism $\frak g(A) \cong \frak g(A^\prime)$.
\item If $A$ has the form $\begin{pmatrix} B &  0 \\ 0 &  C \end{pmatrix}$, 
there is a 
natural isomorphism $\frak g(A) \cong \frak g(B) \times \frak g(C)$. 
\end{enumerate}
\end{proposition}

 \section{The Monster Lie algebra}

 Let $V_{1,1}$ denote the vertex  algebra  for the even unimodular  2-dim Lorentzian lattice $\textrm{II}_{1,1}=\Z\oplus\Z$ equipped with the bilinear form given by the matrix 
$\left(\begin{smallmatrix}  ~0 & -1\\-1 & ~0
 \end{smallmatrix}\right)$. We let $V=V^\natural\otimes V_{1,1}$.
A vector $v\in V$ is called {\it primary}  if it satisfies 
$$L(j)v=0$$ for any $j> 0$, where $L(j)$ are the Virasoro generators. In addition, if
$$L(0)v=nv$$
then $v$ is primary of weight $n$. 
We have $\mathfrak m=P_1/R$  where
$$P_1=\{\psi\in V^\natural\otimes V_{1,1}\mid L(0)\psi=\psi,\ L(j)\psi=0, \ j> 0\},$$
which is the subspace of $V$ of primary vectors of weight 1, $L(j)$ are the Virasoro generators and $R$ denotes the radical of a natural bilinear form on $P_1$. The space $P_1$ is known as the {\it physical space} of $V$, since it contains the vectors in $V$ that correspond to physical states in a conformal field theory. 


We identify the root lattice of $\mathfrak{m}$ with $\textrm{II}_{1,1}$.
Then $\mathfrak{m}$ has root space decomposition (see \cite{BoInvent}, \cite{JLW})
\begin{equation*}\label{mdecomp}\mathfrak{m}=\left(\bigoplus_{m,n<0}\mathfrak{m}_{(m,n)} \oplus \mathfrak{m}_{(-1,1)}\right)\oplus \mathfrak{m}_{(0,0)}\oplus \left(\mathfrak{m}_{(1,-1)}\oplus\bigoplus_{m,n>0}\mathfrak{m}_{(m,n)}\right),\end{equation*}
where $(m,n)\in \textrm{II}_{1,1}$.  Recall \cite{FLM} that  $\M$ acts on $V^\natural$ by vertex operator algebra automorphisms and that $\Aut(V^\natural)$ equals $\M.$
The action of $\M$ on $V^\natural$ induces an action of $\M$ on $\frak m$. In particular, the $\M$-action on $\mathfrak{m}$ gives rise to an $\M$-action on each of the root spaces $\mathfrak{m}_{(1,n)}$ and $ \mathfrak{m}_{(-1,-n)}$ and $\M$ acts trivially on $\frak h=\mathfrak{m}_{(0,0)}$, which is the Cartan subalgebra of $\frak m$  (\cite{BoInvent}, \cite{JLW}).

Borcherds \cite{BoInvent} used the No-ghost Theorem    to obtain an $\M$-module isomorphism between homogeneous subspaces  $\frak m_{(m,n)}$ for $(m,n)\neq (0,0)$ of the Monster Lie algebra and the the homogeneous subspace $V^\natural_{mn+1}$ of $V^\natural$, where $\frak h=\mathfrak{m}_{(0,0)}\cong \mathbb{R}\oplus \mathbb{R}$ is a trivial $\mathbb{M}$-module. %


\subsection{The Monster Lie algebra as $\mathfrak m =\mathfrak g(A) / \mathfrak z$}

As shown in \cite{BoInvent}, the Monster Lie algebra $\mathfrak{m}=P_1/R$ is isomorphic to the  generalized Kac--Moody
 algebra $\mathfrak g(A) / \mathfrak z$,  where $\mathfrak{z}$ is the center of~$\mathfrak g(A)$.
 
In \cite{JurJPAA}, an explicit Borcherds Cartan matrix $A$ was given:
\begin{equation}\label{mdec}%
{A=\smaller \begin{blockarray}{cccccccccc}
 & & \xleftrightarrow{c(-1)}  &   \multicolumn{3}{c}{$\xleftrightarrow{\hspace*{0.7cm}c(1)\hspace*{0.7cm}}$}   &  \multicolumn{3}{c}{$\xleftrightarrow{\hspace*{0.7cm} { c(2)}\hspace*{0.7cm}}$}   & \\
\begin{block}{cc(c|ccc|ccc|c)}
  &\multirow{1}{*}{$c(-1)\updownarrow$} & 2 & 0 & \dots & 0 & -1 & \dots & -1 & \dots \\ \cline{3-10}
    &\multirow{3}{*}{ $ \,\,c(1)\,\,\left\updownarrow\vphantom{\displaystyle\sum_{\substack{i=1\\i=0}}}\right.$}& 0 & -2 & \dots & -2 & -3 & \dots & -3 &   \\
      & & \vdots & \vdots & \ddots & \vdots & \vdots & \ddots & \vdots & \dots  \\
        & & 0 & -2 & \dots & -2 & -3 & \dots & -3 &   \\ \cline{3-10}
         & \multirow{3}{*}{ $\,\,c(2)\,\, \left\updownarrow\vphantom{\displaystyle\sum_{\substack{i=1\\i=0}}}\right.$}  & -1 & -3 & \dots & -3 & -4 & \dots & -4 &   \\
    && \vdots & \vdots & \ddots & \vdots & \vdots & \ddots & \vdots & \dots  \\
        && -1 & -3 & \dots & -3 & -4 & \dots & -4 &   \\ \cline{3-10}
         && \vdots &  & \vdots &  &  & \vdots & \vdots &   \\
\end{block}
\end{blockarray}}\;\;,
\end{equation}
 The numbers $c(j)$ are the coefficients of $q$ in the  modular function 
$$J(q)=j(q) -744=\sum_{i\geq -1}c(j)q^j= \frac1q + 196884 q + 21493760 q^2 + 864299970 q^3  + \cdots$$
so that  $c(-1) = 1$, $c(0) = 0$, $c(1) = 196884$, $\dots$.

As a block matrix, this is:
\begin{equation*}\label{mdec}
{A=\smaller \begin{blockarray}{cccccccccc}
 & &     \multicolumn{3}{c}{}   &  \multicolumn{3}{c}{}   & \multicolumn{3}{c}{}  \\
\begin{block}{cc(c|c|c|c|ccc|c)}
  & & 2_{c(-1)\times c(-1)} & 0_{c(-1)\times c(1)}  & -1_{c(-1)\times c(2)}  & \dots \\ \cline{3-10}
    & & 0_{c(1)\times c(-1)} & -2_{c(1)\times c(1)}   & -3_{c(1)\times c(2)}   &  \dots \\ \cline{3-10}
         & & -1_{c(2)\times c(-1)}  & -3_{c(2)\times c(1)}   & -4_{c(2)\times c(2)}  &  \dots \\ \cline{3-10}
                && \vdots & \vdots & \vdots &   \vdots & &   \\
\end{block}
\end{blockarray}}\;\;,
\end{equation*}
Thus $$ A =  (a_{jk,pq})_{(j,k),(p,q)\in {I}}$$ where $a_{jk,pq} = -(j+p)$.
Thus $\left( a_{jk,pq} \right)$ is a block of size $c(j)\times c(p)$ in block  position $(j,p)$.

We define the index set 
$$I = \left\{(j,k)\mid j\in\{-1,1,2,3,\dots\},\; 1 \leq k\leq c(j) \right\} $$
to reflect the block form of~$A$,
so that
$$ A = \left( a_{jk,pq} \right)_{(j,k),(p,q)\in {I}}$$ where $a_{jk,pq} = -(j+p)$.
Thus $\left( a_{jk,pq} \right)$ is a block of size $c(j)\times c(p)$ in row/column block position $(j,p)$.

Then $\a_{-1}:=\a_{-11}$ is a real root with squared norm $2$, while $\a_{jk}:=\a_{jk,jk} $ is an imaginary root
with squared norm $-2j$ for $(j,k)\in I$ with $j>0$, so for $\N=\{1,2,3,\dots\}$ we define
\begin{align*}
I^{\re}&:= \left\{ (j,k)\in I \mid a_{jk,jk}>0 \right\}=\left\{(-1,1) \right\}, \text{ and} \\
I^{\im} &:= \left\{ (j,k)\in I \mid a_{jk,jk}\le0 \right\}= \left\{(j,k)\mid j\in\N,\; 1 \leq k\leq c(j) \right\}.
\end{align*}

The row space of $A$ is spanned by the two rows indexed by~$(\pm1,1)$.

$$
R_1=(2,0,\dots 0,-1,\dots ,-1,\dots)$$

$$R_2=(0,-2,\dots ,-2,-3,\dots ,-3,\dots)
$$
Then $$(-1/2 R_1+ 3/2 R_2)=(-1,-3,\dots,-3,-4,\dots,-4,\dots).$$
So $A$ has rank 2.

The \emph{Serre--Chevalley generators} of $\mathfrak{g}(A)$ are \cite{JurJPAA}
$e_{jk}$, $f_{jk}$, $h_{jk}$ for all $(j,k)\in  I$, with
 defining relations 
\begin{align*}
\tag{R:1}\label{Rmhh} \left[h_{jk},h_{pq}\right]&=0,\\
\tag{R:2}\label{Rmhe} \left[h_{jk},{e}_{pq} \right]&=a_{jk,pq}{e}_{pq} =-(j+p) {e}_{pq},\\
\tag{R:3}\label{Rmhf} \left[h_{jk},{f}_{pq} \right]&=-a_{jk,pq}{e}_{pq} =(j+p) {f}_{pq},\\
\tag{R:4}\label{Rmef} \left[{e}_{jk},{f}_{pq} \right]&=\delta_{jp}\delta_{kq}h_{jk},\\ 
\tag{R:5}\label{Rmee} \left(\ad {e}_{-1\,1} \right)^j \,{e}_{jk}&= \left(\ad {f}_{-1\,1} \right)^j \,{f}_{jk}=0,
\end{align*}
for all $(j,k),\,(p,q) \in {I}$. 
From here on, we usually write $e_{-1}:={e}_{-1\,1}$ and $f_{-1}:={f}_{-1\,1}$.

The \emph{Cartan subalgebra} of $\mathfrak{g}(A)$ is $\mathfrak{h}_{A}=\sum_{(j,k)\in I} \C h_{jk}$.
The \emph{Cartan involution} $\eta:\mathfrak{g}(A)\to\mathfrak{g}(A)$  acts as $-1$ on $\frak h_A$ and interchanges~$e_{jk}$ and~$f_{jk}$. 

The \emph{Monster Lie algebra} is  $\mathfrak{m}=\mathfrak{g}(A)/\frak{z}$, where $\frak z$ is the center of $\frak g(A)$.
Note that $\mathfrak{z}$ is contained in $\mathfrak{h}_A$ and so the Cartan subalgebra of $\mathfrak{m}$ is 
$ \mathfrak{h}:= \mathfrak{h}_A/\mathfrak{z}$.
The Cartan involution induced on $\mathfrak{m}$ is also denoted by~$\eta$. 
The spaces  $\mathfrak{n}^\pm$  intersect $\mathfrak{h}$ trivially, so
we can identify them with subsets of~$\mathfrak{m}$.
The matrix $A$ has rank~2 and so $\mathfrak{h}$ has dimension~2.

Define the following elements in $\frak m = \frak g(A)/\frak z$:
\begin{align*}
h_1 &:= \frac12(h_{-1\,1}-h_{1\,1})+\mathfrak{z}, & h_2 &:= -\frac12(h_{-1\,1}+h_{1\,1})+\mathfrak{z},\\
e_{-1} &:= {e}_{-1\,1}+\mathfrak{z}, & f_{-1} &:= {f}_{-1\,1}+ \mathfrak{z},
\end{align*}
 and write $e_{jk}$ for $e_{jk}+\mathfrak{z}$ and $f_{jk}$ for ${f}_{jk}+ \mathfrak{z}$, for  all $(j,k)\in  I-\left\{(-1,1) \right\}$. 
The following proposition now gives explicit generators and relations for $\mathfrak{m}=\frak g(A)/\frak z$.
\begin{proposition}\label{P-Monster} \cite{CJM}
The Serre--Chevalley generators of $\mathfrak{m}$ are
$e_{-1}$, $f_{-1}$, $h_{1}$, $h_2$, and $e_{jk}$, $f_{jk}$ for all\newline
$(j,k)\in  I-\left\{(-1,1) \right\}$,
with defining relations:
\begin{align*}
\tag{M:1}\label{Mhh} \left[h_{1},h_2\right]&=0,\\
\tag{M:2a}\label{Mhe-} \left[h_1,e_{-1} \right] &= e_{-1},&              
              \left[h_2, e_{-1} \right] &= -e_{-1},\\
\tag{M:2b}\label{Mhe} \left[h_1,e_{jk} \right]&= e_{jk},&
              \left[h_2, e_{jk} \right] &= j e_{jk},\\
\tag{M:3a} \label{Mhf-} \left[h_1,f_{-1} \right] &= -f_{-1},&
             \left[h_2,f_{-1} \right] &= f_{-1},\\
\tag{M:3b}\label{Mhf} \left[h_1,f_{jk} \right]&= - f_{jk},&
              \left[h_2,f_{jk} \right] &= -j f_{jk},\\
\tag{M:4a}\label{Me-f-} \left[e_{-1},f_{-1} \right]&=h_1-h_2, \\
\tag{M:4b}\label{Mef-} \left[e_{-1},f_{jk} \right]&=0,  & \left[e_{jk},f_{-1} \right]&=0,\\
\tag{M:4c}\label{Mef} \left[e_{jk},f_{pq} \right]&=-\delta_{jp}\delta_{kq} \left(jh_1 + h_2\right),\\ 
\tag{M:5}\label{Mee}
  \left(\ad e_{-1} \right)^j e_{jk}&=0,&\qquad \left(\ad f_{-1} \right)^j f_{jk}&=0,
\end{align*}
for all $(j,k),\,(p,q) \in  I-\left\{(-1,1) \right\}$. Also ${\frak h} = \C h_1\oplus \C h_2 = \mathfrak{h}_A/\mathfrak{z}$, the Cartan subalgebra of~$\mathfrak{m}$.
\end{proposition}

\begin{figure}[!ht]
\includegraphics[height=11.0cm]{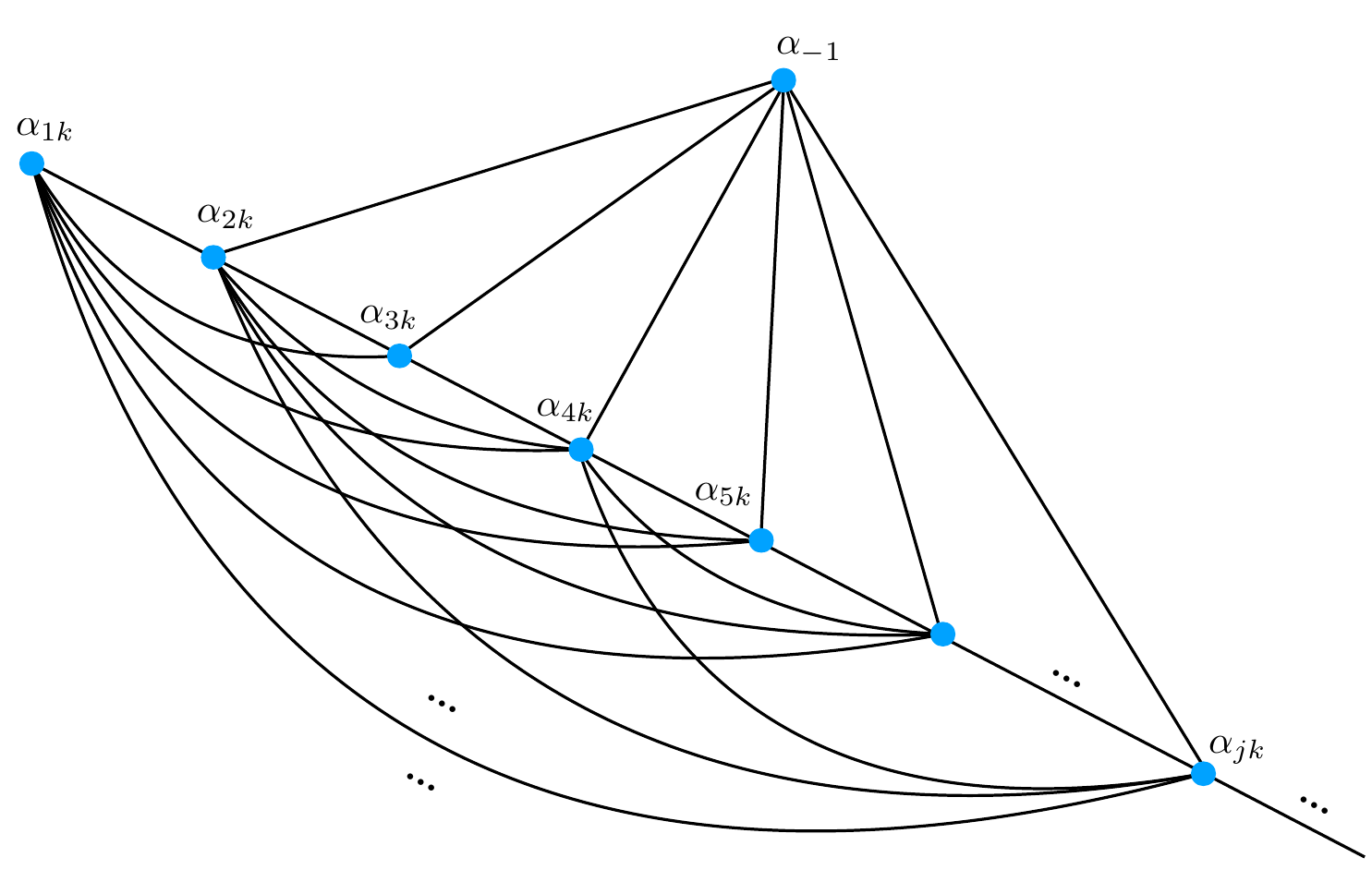} 
\caption{The Dynkin diagram for the Monster Lie algebra $\frak m$ with real simple root $\alpha_{-1}$ and imaginary simple roots $\alpha_{jk}$ is a complete graph on a countable set, missing only one egde, namely the edge between $\alpha_{-1}$ and $\alpha_{1k}$. The multiplicities on the edges are omitted. Edges between $\alpha_{-1}$ and $\alpha_{jk}$ have multiplicity  $|1-j|$ and edges between $\alpha_{jk}$ and $\alpha_{j'k'}$ have multiplicity $j+j'$.  Each vertex $\alpha_{jk}$ represents a complete graph on $c(j)$ vertices $\{\a_{jk}\mid 1\leq k\leq c(j)\}$, where each edge has multiplicity $2j$. 
}
\end{figure}

We have $\mathfrak{m}=\mathfrak{g}(A)/\mathfrak{z}$ and  the center $\mathfrak{z}$ is spanned by
$$(p-j)h_{-1\,1}-(p+1)h_{jk}+(j+1)h_{pq}.$$
So 
$h_{-1\,1}+\mathfrak{z}=h_1-h_2$; 
$\;h_{1\,k} + \mathfrak{z}=-h_1-h_2$ 
for $(1,k)\in I$; and 
\begin{align*}
h_{jk} + \mathfrak{z} &= h_{jk} + \frac{1}{2}\left((1-j)h_{-1\,1}-(1+1)h_{jk}+(j+1)h_{1\,1} \right) +\mathfrak{z} \\
&=\frac{1-j}{2}\,h_{-1\,1} +\frac{1+j}{2}\, h_{1\,1} +\mathfrak{z}\\
& = \frac12\left(h_{-1\,1}+h_{1\,1} \right) + \frac{j}2\left(-h_{-1\,1}+h_{1\,1} \right)+{\frak z}\\
&=-jh_1-h_2
\end{align*}
when $j\ne 1$. Hence $h_1$ and $h_2$ span ${\frak h}$. Moreover $h_1$ and $h_2$ are linearly independent by \eqref{Mhe-} and \eqref{Mhe}.
\begin{figure}[!ht]
\begin{center}
\includegraphics[height=14.0cm]{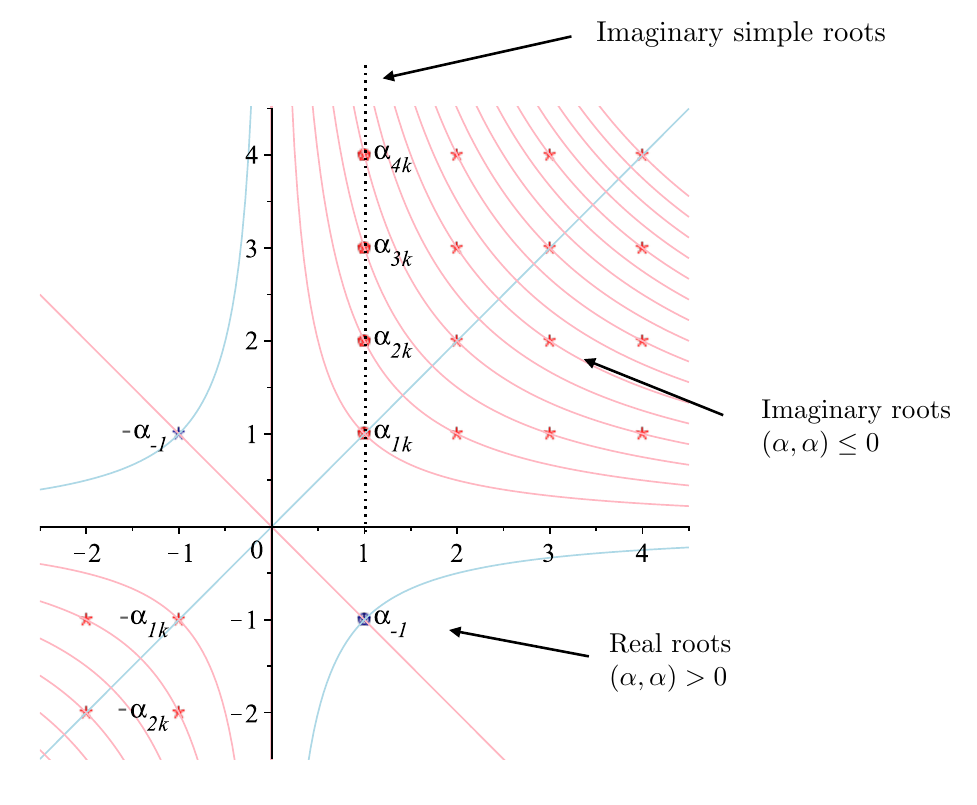} 
\caption{Root lattice of the Monster Lie algebra $\mathfrak m$}\label{F-rts}
\end{center}
 \end{figure}

\section{Decomposition of  $\mathfrak m$}

In \cite{JurJPAA}, Jurisich obtained a decomposition of the Monster Lie algebra  into two free subalgebras $ \mathfrak{u}^\pm $ and a copy of $ \mathfrak{gl}_2 \cong \mathfrak{sl}_2 \oplus \mathbb{C}h $,
where  $\mathfrak{u}^\pm$ are free Lie algebras generated by countably many elements of the form
$$ e_{\ell,jk}:= \left(\ad e_{-1} \right)^\ell e_{jk}\quad\text{and}\quad f_{\ell,jk}:= \left(\ad f_{-1} \right)^\ell f_{jk}
$$
respectively. The real simple root vectors $e_{-1}$, $f_{-1}$ and the imaginary simple  root vectors $e_{jk}$, $f_{jk}$ (which correspond to the case $\ell=0$ in the notation for the free subalgebras $\mathfrak u^\pm$) are the Chevalley generators of $\mathfrak m$.

Furthermore, the  free subalgebras $\mathfrak{u}^\pm$ are freely generated by $ \mathfrak{gl}_2 $-modules, given by the adjoint action. The proof of the existence of such a decomposition uses the explicit form of the Borcherds Cartan matrix for $\mathfrak m$, also given in \cite{JurJPAA}. 

This decomposition led to a 
simplification in \cite{JLW} of Borcherds' proof of the Monstrous Moonshine Conjecture, particularly in the proof of the denominator identity for $\mathfrak m$.
In particular, the decomposition simplified the computation of the needed semi-infinite homology of $ \mathfrak{u}^+ $, and hence the character recursions.
Together with $ \mathfrak{sl}_2 $-representation theory, this allowed \cite{JLW} to efficiently compute the  denominator formula for the Monster Lie algebra.


\section{Monstrous Lie algebras}

\subsection{Monstrous Lie algebras of  Fricke type}
An element $g\in\mathbb{M}$ is called {\it Fricke} if  the McKay--Thompson series $T_{g}(\tau)$ is invariant
under the level $N$ Fricke involution $\tau \mapsto -1/N\tau$ for some $N\geq 1$. Otherwise $g$ is called non-Fricke. There are 141 Fricke classes, and 53 non-Fricke classes in $\M$. When $g$ is Fricke, $\mathfrak m_g$ has similar structure to the
Monster Lie algebra $\mathfrak m$.

The Lie algebra $\mathfrak m_g$, for $g$ of Fricke type, has denominator formula \cite{CarDuke}:
$$f(p) - f(q^{1/N}) = p^{-1} \prod_{m >0, \; n \in \frac{1}{N}\Z} (1-p^m q^n)^{c(m,n)}$$ where $$f(q) = q^{-1} + \sum_{n=1}^\infty c(1,\frac{n}{N}) q^n.$$
Note that in \cite{CarFricke}, the denominator formula given in Proposition 3.4 has  a different variable normalization and a $\Z\times\Z$-grading instead of a $\Z\times\frac{1}{N}\Z$-grading. The present exposition uses the $\Z\times\frac{1}{N}\Z$-grading.



The exponents $c(m,\frac{n}{N})$ are the multiplicities of roots $(m,\frac{n}{N})\in\Z\oplus \frac{1}{N}\Z$. 
They are  the coefficients of a discrete Fourier transform of the generalized McKay--Thompson series relative to the cyclic group $\langle g\rangle$ \cite{CarDuke}.

For $g$ of Fricke type, we have $\mathfrak m_g=\mathfrak{g}(A_g)/\mathfrak z$ where $\mathfrak z$ is the center of $ \mathfrak{g}(A_g)$ and $A_g$ is the Borcherds Cartan matrix (Carnahan):

\begin{equation*}
{A_g=\smaller \begin{blockarray}{cccccccccc}
 & & \xleftrightarrow{c(1,\frac{-1}{N})}  &   \multicolumn{3}{c}{$\xleftrightarrow{\hspace*{0.7cm}c(1,\frac{1}{N})\hspace*{0.7cm}}$}   &  \multicolumn{3}{c}{$\xleftrightarrow{\hspace*{0.7cm} { c(1,\frac{2}{N})}\hspace*{0.7cm}}$}   & \\
\begin{block}{cc(c|ccc|ccc|c)}
  &\multirow{1}{*}{$c(1,\frac{-1}{N})\updownarrow$} & 2 & 0 & \dots & 0 & -1 & \dots & -1 & \dots \\ \cline{3-10}
    &\multirow{3}{*}{ $ \,\,c(1,\frac{1}{N})\,\,\left\updownarrow\vphantom{\displaystyle\sum_{\substack{i=1\\i=0}}}\right.$}& 0 & -2 & \dots & -2 & -3 & \dots & -3 &   \\
      & & \vdots & \vdots & \ddots & \vdots & \vdots & \ddots & \vdots & \dots  \\
        & & 0 & -2 & \dots & -2 & -3 & \dots & -3 &   \\ \cline{3-10}
         & \multirow{3}{*}{ $\,\,c(1,\frac{2}{N})\,\, \left\updownarrow\vphantom{\displaystyle\sum_{\substack{i=1\\i=0}}}\right.$}  & -1 & -3 & \dots & -3 & -4 & \dots & -4 &   \\
    && \vdots & \vdots & \ddots & \vdots & \vdots & \ddots & \vdots & \dots  \\
        && -1 & -3 & \dots & -3 & -4 & \dots & -4 &   \\ \cline{3-10}
         && \vdots &  & \vdots &  &  & \vdots & \vdots &   \\
\end{block}
\end{blockarray}}\;\;,
\end{equation*}



Let $$I = \left\{(j,k)\mid j\in\{-1,1,2,3,\dots\},\; 1 \leq k\leq c(1,j/N) \right\}. $$

For $(j,k)\in {I}$, the entries satisfy
$$ (\a_{jk},\a_{pq})=a_{jk,pq} = -(j+p). $$

In the case where $g$ is Fricke, Carnahan \cite{CarFricke} proved an analog for $\mathfrak{m}_g$ of Jurisich's decomposition $\mathfrak m\cong\mathfrak u^-\oplus\mathfrak{gl}_2\oplus\mathfrak u^+$ where $\mathfrak u^\pm$ are free Lie algebras. 
The  $ \mathfrak{gl}_2 $-summand occurs since the Fricke type monstrous Lie algebras have exactly one real simple root. Carnahan also showed that the Borcherds Cartan matrix for $\mathfrak{m}_g$ has the same form as that of $\mathfrak{m}$, with different block sizes as given by the dimensions of the root spaces for imaginary simple roots of $\mathfrak{m}_g$.
With this decomposition and homology computations similar to those in \cite{GarLep}, Carnahan  simplified the computation of the twisted denominator formulas for $\mathfrak{m}_g$.

\subsection{Generators and relations for $\mathfrak{m_g}$}
For all $(j,k)\in  I-\left\{(-1,1) \right\}$, the generators of $\mathfrak{m_g}$ are
$e_{-1}$, $f_{-1}$, $h_{1}$, $h_2$, and $e_{jk}$, $f_{jk}$ with defining relations:
\begin{align*}
 \left[h_{1},h_2\right]&=0,\\
 \left[h_1,e_{-1} \right] &= e_{-1},&              
              \left[h_2, e_{-1} \right] &= -e_{-1},\\
 \left[h_1,e_{jk} \right]&= e_{jk},&
              \left[h_2, e_{jk} \right] &= j e_{jk},\\
 \left[h_1,f_{-1} \right] &= -f_{-1},&
             \left[h_2,f_{-1} \right] &= f_{-1},\\
 \left[h_1,f_{jk} \right]&= - f_{jk},&
              \left[h_2,f_{jk} \right] &= -j f_{jk},\\
 \left[e_{-1},f_{-1} \right]&=h_1-h_2, \\
 \left[e_{-1},f_{jk} \right]&=0,  & \left[e_{jk},f_{-1} \right]&=0,\\
 \left[e_{jk},f_{pq} \right]&=-\delta_{jp}\delta_{kq} \left(jh_1 + h_2\right),\\ 
  \left(\ad e_{-1} \right)^j e_{jk}&=0,&\qquad \left(\ad f_{-1} \right)^j f_{jk}&=0,
\end{align*}
The Lie algebra $\mathfrak m_g$ has Cartan subalgebra${\frak h} = \C h_1\oplus \C h_2 = \mathfrak{h}_A/\mathfrak{z}$.


\subsection{Structure theorem for $\mathfrak m_g$ for $g$ of Fricke type}

 The unique real simple root is $\a_{-1}=(1,\frac{-1}{N})$. 
 The imaginary simple roots are 
 $\a_{jk}=(1,\frac{j}{N})$ for $j\in\Z_{>0}$.  
 
 The positive imaginary roots are $(m,\frac{j}{N})$ for $m,j\in\Z_{>0}$ which have norm $\frac{-2mj}{N}$.

The following theorem follows immediately from Jurisich, JPAA, Corollary 5.1 (also noted by Carnahan who called such Lie algebras {\it Fricke Lie algebras}).
\begin{theorem} Let $g\in\M$ be of Fricke type. 
Let $I=\{1,2,3,\dots\}$ and $I_0=\{-1\}$. Then 
$$\mathfrak{m}_g = \mathfrak u^+ \oplus \mathfrak{gl}_2 \oplus \mathfrak {u}^-$$ where $\mathfrak u^+$ is the free Lie algebra  $L(S^+)$  on the the set 
$$S^+=\{({\ad} e_i  )^{\ell } e_j \mid i \in I_0,\ j\in I,\ \ 0\leq \ell\leq j-1 \},$$
and $\mathfrak u^-$ is the free Lie algebra $L(S^-)$ on the set
$$S^-=\{({\ad} f_i  )^{\ell } f_j \mid i \in I_0,\ j\in I,\ 0\leq \ell\leq j-1 \}.$$
\end{theorem}

\subsection{Example: the Fricke Lie algebra $\mathfrak m_{2A}$, the Baby Monster } 

 The Baby Monster finite simple group $\mathbb B$ contains Fricke 
 elements of $\M$ and coincides with conjugacy class 2A in $\M$. In this case, $N=2$.

The McKay--Thompson series  $T_{2A}(\tau)$ for $\mathbb B$ 
 with constant term $a(0) = 104$ is given by
\begin{align*}j_{2A}(\tau)
&=T_{2A}(\tau)+104\\
&=\left(\left(\tfrac{\eta(\tau)}{\eta(2\tau)}\right)^{12}+2^6 \left(\tfrac{\eta(2\tau)}{\eta(\tau)}\right)^{12}\right)^2\\
&=\frac{1}{q} + 104 + 4372q + 96256q^2 +1240002q^3+10698752q^4+\dots
\end{align*}
and $\eta(\tau)$ is the Dedekind eta function.
The root multiplicities $c(1,\frac{n}{2})$ for the simple roots $(1,\frac{n}{2})$ are the coefficients of $q^n$ in 
$$T_{2A}(-1/\tau)=T_{2A}(\tau/2).$$  Thus
$$c(1,\frac{1}{2})=4372,\ c(1,1)=96256,\ c(1,\frac{3}{2})=1240002,\ c(1,2)=10698752,\dots$$



\begin{figure}[!ht]
\includegraphics[height=11.5cm]{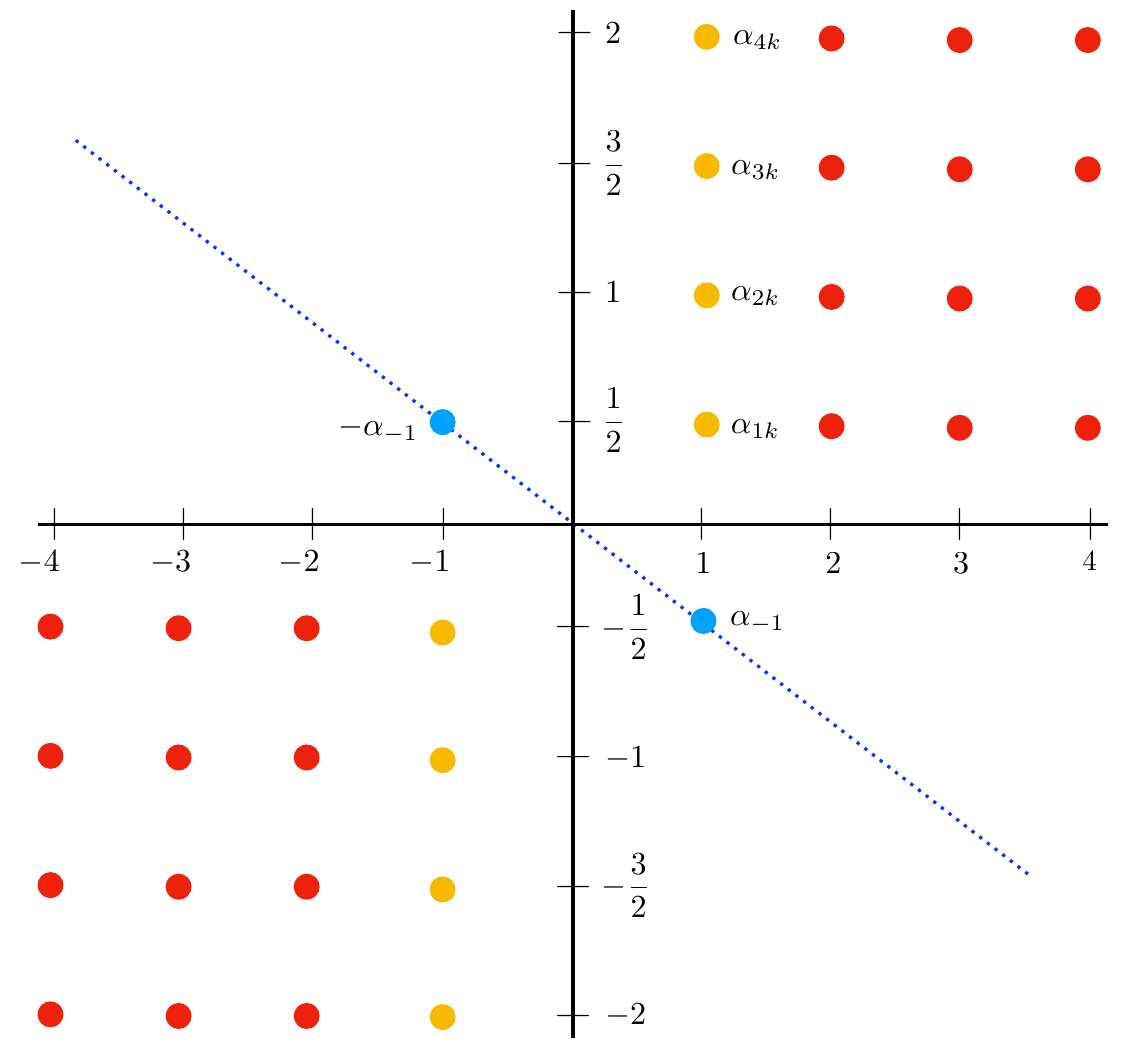} 
\caption{Root lattice of the Baby Monster Lie algebra $\mathfrak m_{2A}$}\label{F-rts}
 \end{figure}

\begin{figure}[!ht]
\includegraphics[height=11.5cm]{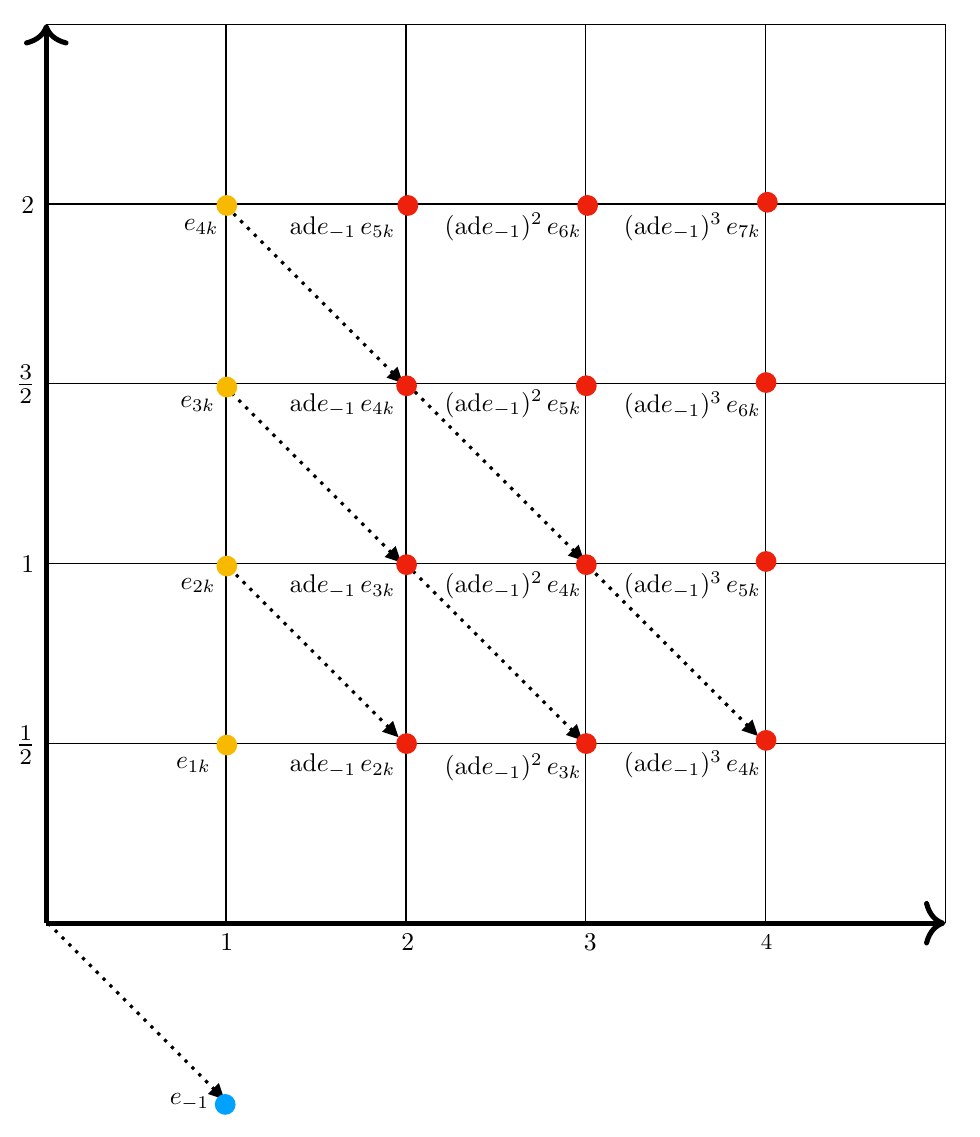} 
\caption{$\a_{-1}$-weight strings of $\mathfrak{gl}_2(-1)$-modules in the Baby Monster Lie algebra $\mathfrak m_{2A}$}
 \end{figure}

\subsection{Example: the non-Fricke Lie algebra $\mathfrak m_{2B}$}

Every non-Fricke  $T_g$ has a preferred eta-product expansion (Proposition 4.2, \cite{Car51}):
$$T_{g^{N/n}}=\prod_{i=1}^k\eta(a_i\tau)^{b_i}$$

where the Dedekind $\eta$-function is 
$$\eta(\tau) = q^\frac{1}{24} \prod_{n=1}^\infty \left(1 - q^n\right).$$

The $\eta$-quotient expression for $T_{2B}(\tau)$ is ([FLM], Section 10.5) $T_{2B}(\tau)=\dfrac{\eta(\tau)^{24}}{\eta(2\tau)^{24}}$. We now apply the transformation $\tau\to(-1/\tau)$ under which
$$\eta(\tau)\to\eta\left(-\frac{1}{\tau}\right) = \sqrt{-i\tau}\, \eta(\tau).$$

The transformed $\eta$-product is 
$$T_{2B}(-1/\tau)={2^{12}}\dfrac{\eta(\tau)^{24}}{\eta\left(\frac{\tau}{2}\right)^{24}}.$$
This gives
\begin{align*}T_{2B}(-1/\tau)&= 4096q^{1/2} + 98304q + 1228800q^{3/2}+\dots\\
&=2^{12}(q^{1/2}+24q+300q^{3/2}+\dots)
\end{align*}



We note that $$\dfrac{1}{T_g\left(\dfrac{\tau}{2}\right)}=q^{1/2}+24q+300q^{3/2}+\dots$$
so we obtain the identity (also observed in \cite{DLiM})
$$T_{2B}(-1/\tau)=\dfrac{2^{12}}{T_g(\dfrac{\tau}{2})}$$
which is of the form 
$$T_{g}(-1/\tau)=k_1+\dfrac{k_2}{T_g(\dfrac{\tau}{N})}$$
for some non-negative integers $k_1$ and $k_2$, as suggested by Carnahan. We note the similarity with the Jacobi theta function
$$\theta_4(\tau)=\dfrac{\eta\left(\frac{\tau}{2}\right)^2}{\eta(\tau)}=\sum_{n\in\Z}(-1)^ne^{\pi i n^2\tau}.$$
 \begin{lemma}The transformed series $T_{2B}(-1/\tau)$ for conjugacy class 2B is
$$T_{2B}(-1/\tau)={2^{12}}\dfrac{\eta(\tau)^{24}}{\eta\left(\frac{\tau}{2}\right)^{24}}
=\dfrac{2^{12}}{T_{2B}(\frac{\tau}{2})}=\left(\dfrac{2\eta(\tau)}{\theta_4(\tau)}\right)^{12}.
$$
\end{lemma}
This gives an example of the claim of Conway and Norton \cite{CN} that $\eta$-products can be written in terms of $\theta$-functions.





To obtain root multiplicities for $\mathfrak m_{2B}$, note that

\begin{align*}
T_{2B}(\tau)&=\dfrac{\eta(\tau)^{24}}{\eta(2\tau)^{24}}\\
&=\dfrac{1}{q}\prod_{n>0} (1-q^{2n})^{24}(1-q^{2n+1})^{24}(1-q^{2n})^{-24}\\
&=\dfrac{1}{q}\prod_{n\geq 0} (1-q^{2n+1})^{24}\\
&=\dfrac{1}{q}\prod_{m\geq 0} (1-q^{m})^{c(m,0)}.
\end{align*}
The root multiplicities are  then
$$c(m,0)=
\begin{cases} 
0,&\text{$m$ even}\\ 
24,&\text{$m$ odd}\\ 
\end{cases}$$
This agrees with the values Carnahan's table in \cite{Car51}.

We have
$$T_{2B}(-1/\tau)= 4096q^{1/2} + 98304q + 1228800q^{3/2}+\cdots.$$
Thus the $c(1,\frac{n}{2})$are given by
$$c(1,\frac{1}{2})=4096,\ c(1,1)=98304,\ c(1,\frac{3}{2})=1228800,\cdots.$$

\subsection{Lie algebraic structure of $\mathfrak m_g$}

 The  Lie algebras $\mathfrak m_g$ are constructed in terms of  semi-infinite
cohomology  in \cite{CarDuke} and \cite{DF}. An earlier  construction of  $\mathfrak m_g$ by Borcherds and independently by I. Frenkel was in terms of the No-ghost Theorem applied to  the 26D bosonic string (see \cite{BoInvent}, \cite{DF} and \cite{FLM}).

In joint work with H. Chen, J. Du, D. Hou,  D. Tan and  F. Thurman \cite{CCDHTT}, we studied the structure of the Lie algebras $\mathfrak m_g$ as Borcherds algebras via their Borcherds Cartan matrices, generators and relations and their root systems (see also \cite{Tan}, this volume). We proved a variation of Jurisich's theorem which shows that like the Monster Lie algebra $\mathfrak m$, when $g$ is non-Fricke,  $\mathfrak m_g$ contain free subalgebras generated by certain positive and negative imaginary root vectors.  This was known in the Fricke case \cite{Car51}.

If $g\in\M$ is non-Fricke, then 
$$\mathfrak m_g= \mathfrak{u}^- \oplus (\mathcal{H} \oplus \mathbb{C} h_{(1,1)} ) \oplus \mathfrak{u}^+$$
where $\mathfrak{u}^\pm$ are free Lie algebras generated by a countable set of certain root vectors, 
$ \mathcal{H} $ is an infinite-dimensional Heisenberg Lie algebra and $h_{(-1,1)} $ is a central element of $ \mathcal{H} $.


\section{Constructing Lie group analogs for the Lie algebras $\mathfrak m_g$}

Since the Monstrous Lie algebras $\mathfrak m_g$ describe physical symmetries, it is important to associate an analog of a Lie group $G(\mathfrak m_g)$ to them. There are known difficulties with constructing groups for Borcherds  algebras,  since root vectors  corresponding to  imaginary simple roots do not act locally nilpotently on the adjoint representation, or on most faithful highest weight modules. 

To overcome this, several different but interrelated constructions have been developed (\cite{CJM}, \cite{CJM2026}, \cite{CJ}). As is the case with Kac--Moody algebras, it turns out to be natural and useful to consider a (topological) completion of the Borcherds algebra and to construct a group of automorphisms of the completion.

In \cite{CJM2026}, the authors constructed a complete pro-unipotent group $\widehat{U}^+$ of automorphisms of a completion $\widehat{\frak m}$ of $\frak m$ by defining  the algebraic notion of \emph{pro-summability} which is equivalent to convergence on finite dimensional subspaces.
The generators of $\widehat{U}^+$  as a topological group are 
 power series with constant term 1. 
The authors show that the action of the Monster finite group $\M$ on $\mathfrak m$ induces an action of~$\M$ on~$\widehat{\frak m}$, and that this induces a compatible  action of $\M$ on~$\widehat{U}^+$.

In \cite{CJM}, the authors also constructed a group $G(\frak m)$ associated with all roots of  $\mathfrak m$, both real and imaginary, and relations between them. In contrast, Tits' presentation of a Kac--Moody group \cite{Ti87}  only has generators and relations associated to real roots.   

The group $G(\frak m)$ is not a group of automorphisms of $\mathfrak m$ itself,  but the authors show that 
certain subgroups of $G(\mathfrak m)$ are automorphism groups of  $\mathfrak{m}$, of certain $\gl_2$ subalgebras of $\mathfrak{m}$ and  of a completion $\widehat{\mathfrak{m}}$ of $\mathfrak{m}$.

For the class of Borcherds algebras that admit a direct sum decomposition into a  Kac--Moody (or semisimple) algebra $\mathfrak g_J$ and  free Lie algebras $\mathfrak u^\pm$ as in \cite{JurJPAA}, in joint work with Elizabeth Jurisich \cite{CJ}, the authors   constructed a semi-direct product of a Kac--Moody (or semisimple) group and a Magnus group of invertible formal power series corresponding to the negatives of the imaginary simple roots. We also applied our  construction to a number of concrete examples.

\section{Open problems and future directions}
The open problems in this area are numerous. They connect to group theory, number theory Lie theory and string theory. We mention a few open problems below.

\medskip
\subsection{Lie group analogs for $\mathfrak m_g$}
\leavevmode

{\bf Question 1.} It should  be possible to construct an analog~$G(\frak g)$ of the group~$G(\frak m)$ of \cite{CJM} for any Borcherds algebra~$\frak g$, relative to a Borcherds  Cartan matrix~$A$. Non-zero entries on the diagonal of $A$  corresponding to real or imaginary roots would give rise to $\SL_2$ subgroups of $G(\frak g)$ and these would correspond to groups of automorphisms of the relevant $\frak{sl}_2$ subalgebras of $\frak g$.
Zero entries on the diagonal correspond to norm zero imaginary simple roots. The Lie subalgebra generated by $\widehat{e}_i$, $\widehat{h}_i$ and $f_i$ is a 3-dimensional Heisenberg subalgebra of $\frak g$ which would give rise to 3-dimensional Heisenberg subgroups of $G(\frak g)$.  It should be possible to construct groups of automorphisms of 3-dimensional Heisenberg subalgebras of $\frak g$ corresponding to these subgroups.

{\bf Question 2.} What role does $G(\mathfrak m)$ play in Monstrous Moonshine, if any?

{\bf Question 3.} The action $\M=\Aut(V^\natural)   \to    \Aut(\frak m)$ of $\M$ on $\mathfrak m$ comes from the \cite{FLM}-action of $\M$ on $V^\natural$. 

(a) Does some subgroup of $G(\mathfrak m)$ act on $V^\natural$? 

(b) Are the \cite{FLM} involutions on $V^\natural$ contained in $G(\mathfrak m)$? 

(c) Is the centralizer $C_\M(g)$ a subgroup of  $G(\mathfrak m)$ for any $g\in\M$? For involutions $g\in\M$?

(d) Is $\M$ contained in the centralizer of $G(\mathfrak m)$ or vice-versa? 

(e) Is there a representation of $G(\mathfrak m)$  (or a subgroup of $G(\mathfrak m)$) on ${\rm End}(V^\natural)$?

\bigskip
\subsection{The structure of $\mathfrak m_g$}\leavevmode

{\bf Question 4.} Recall the decomposition of the 
Monster Lie algebra from \cite{JurJPAA}
$$
\mathfrak m= \mathfrak{u}^- \oplus \mathfrak{gl}_2 \oplus \mathfrak{u}^+.
$$
When $g\in\M$ is Fricke, $\mathfrak m_g$ has similar structure \cite{CarFricke}:
$$\mathfrak m_g= \mathfrak{v}^- \oplus \mathfrak{gl}_2 \oplus \mathfrak{v}^+$$
where $\mathfrak{u}^\pm$ and $\mathfrak{v}^\pm$ are free Lie algebras, where $\mathfrak{v}^\pm$ are generated by countably many elements of the form
$$
e_{\ell,jk}:= \left(\ad e_{-1} \right)^\ell e_{jk}\quad\text{and}\quad f_{\ell,jk}:= \left(\ad f_{-1} \right)^\ell f_{jk}.
$$
Carnahan used this decomposition to simplify the proof of the  twisted denominator identities for $\mathfrak m_g$ in the Fricke case. Can the  Monstrous Lie algebras $\mathfrak m_g$ of Fricke type be realized as proper subalgebras of $\mathfrak m$?

{\bf Question 5.} Let $g\in\M$ and let ${\mathfrak m}^g$ be the subalgebra in $\mathfrak m$ of fixed points  under~$g$:
$${\mathfrak m}^g=\{x\in \mathfrak m\mid g\cdot x=x\}\subset \mathfrak m.$$
It is natural to consider ${\mathfrak m}^g$ to be the Lie algebra corresponding to $C_\M(g)$.
 What is the relationship between ${\mathfrak m}^g$ and~$\mathfrak m_g$?

\bigskip
\subsection{Further realization of $\mathfrak m_g$ in string theory} \leavevmode

{\bf Question 6.} Is there a physical interpretation of the  generators of $\mathfrak m=P_1/R$   in (the space of BPS-states of) the model of the compactified  Heterotic string  by Persson, Paquette and Volpato?  These generators, corresponding to simple roots $(1,j)$, are of the form 
$$e_{j,u}=u\otimes \iota({c_j})+R$$
where $u$ is a primary vector in $V^\natural_{j+1}$ and for fixed $j>0$ we fix $c_j\in \widehat{L}$ (a degree 2 central extension of $L=II_{1,1}$)  such that $\bar{c_j}=(1,j)$.

{\bf Question 7.} In joint work with Natalie Paquette \cite{CP}, we gave an interpretation of certain  symmetries of the root system of the Monster Lie algebra $\frak m$, and more generally the  infinite family of monstrous Lie algebras $\frak m_g$,  in terms of discrete symmetries of the compactified model of the Heterotic String by Persson, Paquette and Volpato. 

Can these discrete symmetries, which are subgroups of $O(2,2:\Z)\leq O(2,2:\R),$  be realized in $G(\mathfrak m)$?
Up to an extension by a group of order two, certain of these symmetries (involutions) are contained in $G(\mathfrak m)$.

{\bf Question 8.} How do the free subalgebras of the Monster Lie algebra $\frak m$ (or their generators) arise  from vertex algebra elements in  $V=V^\natural\otimes V_{1,1}$?  Is there a physical interpretation of these free subalgebras and their action in terms of the Heterotic String?

\bigskip

\subsection{Non-Fricke denominator identities}
\leavevmode

{\bf Question 9.} The decomposition  $\mathfrak m_g= \mathfrak{u}^- \oplus (\mathcal{H} \oplus \mathbb{C} h_{(1,1)} ) \oplus \mathfrak{u}^+$ 
was used in \cite{Tan} to simplify Carnahan's proof of the twisted denominator identity for $\mathfrak m_g$ when $g$ is non-Fricke.  Can this decomposition be used to simplify the semi-infinite cohomology construction of $\mathfrak m_g$?


\bibliographystyle{amsalpha}
\bibliography{VOAmath}{}

\end{document}